\documentclass[11pt]{article}
\input{xypic.tex} \xyoption{all}
\usepackage{amsthm,amssymb,latexsym}
\usepackage{graphics}
\usepackage{amsfonts}
\usepackage{amsmath}
\usepackage{verbatim}
\usepackage{amscd}
\usepackage[latin1]{inputenc}
\newcommand{\R}{\mathbb{R}} 
\newcommand{\Z}{\mathbb{Z}}

 \newcommand{\ox}{{\cal O}_X}

\newtheorem{theorem}{Theorem}[section]
\newtheorem{proposition}[theorem]{Proposition}
\newtheorem{lemma}[theorem]{Lemma}

\newtheorem{example}[theorem]{Example}
\newtheorem{definition}[theorem]{{\bf Definition}}

\numberwithin{equation}{section}
\def\lie#1{\mathfrak{#1}}
\def\tlie#1{\tilde{\mathfrak{#1}}}

\def\gb#1{{\mbox{\boldmath $#1$}}}
\def\gb#1{{\mbox{\boldmath $#1$}}}
\def\cal#1{\text{$\mathcal{#1}$}}

\def\opl_#1^#2{\text{\scriptsize$\bigoplus\limits_{\text{\footnotesize$#1$}}^{\text{\footnotesize$#2$}}$}}
\begin{document}

\title{On Moduli Spaces for Abelian Categories}
\author{Vyacheslav Futorny \\ IME - USP \\
Departamento de Matem\'atica \\ Caixa Postal 66281 \\
05315-970 São Paulo-SP, Brazil \\ \\
Marcos Jardim and Adriano Moura \\ IMECC - UNICAMP \\
Departamento de Matem\'atica \\ Caixa Postal 6065 \\
13083-970 Campinas-SP, Brazil}
\date{}

\maketitle

\begin{abstract}
We show that if $\cal A$ is an abelian category satisfying certain mild conditions,
then one can introduce the concept of a moduli space of (semi)stable objects which
has the structure of a projective algebraic variety. This idea is applied to several
important abelian categories in representation theory, like highest weight categories.

\end{abstract}

%---------------------------------------------------------------
%---------------------------------------------------------------

\section{Introduction}

Stability first arose as a geometric notion, within the context of Mumford's Geometric
Invariant Theory. Roughly speaking, if $G$ is an algebraic group acting on an algebraic
variety $X$, a point $x\in X$ is said to be stable with respect to the action of $G$ if
its orbit is closed. In the nineteen sixties, Mumford and others used this geometric notion
to construct moduli spaces of algebraic vector bundles over nonsingular algebraic curves,
translating it into a notion of stability for vector bundles, the so-called
Mumford-Takemoto stability. This very successful theory has been greatly expanded in the
past 40 years by several authors and today the construction of moduli spaces of sheaves
over algebraic varieties is well understood alongside several different variations.
It is also important to mention that moduli spaces of stable vector bundles have also been extremely useful in areas other than algebraic geometry, particularly in mathematical physics.

In 1994, King translated the geometric notion of stability into a concept of stability for finite dimensional modules over finite dimensional associative algebras \cite{K}, and
constructed moduli spaces of representations. This is especially useful in the study of
wild algebras, since the task of describing the structure of indecomposable representations
of such algebras is, from the purely algebraic point of view, hopeless. The best one can
try to achieve is to describe the geometry of the moduli spaces of (semi)stable
representations for a fixed dimension vector. A great deal was accomplished by King
\cite{K} and Schofield \cite{Sch} for the case of hereditary algebras.

In 1997, Rudakov introduced in \cite{Ru} a purely categorical notion of
stability for an object of an abelian category and observed that the Mumford-Takemoto stability for algebraic vector bundles over curves and King's stability for modules
were examples of his categorical notion. It was unclear however how Rudakov's categorical
stability is related to geometric stability in other abelian categories.

The goal of the present paper is to apply King's and Rudakov's
results to the realm of representation theory in the hope that these
techniques  will prove to be as useful as they were in other fields.

In Section \ref{s:catst}, we review Rudakov's categorical stability notion and use it to
construct a special type of stability structure on a wide range of abelian categories.
The central result (Theorem \ref{mthm}) states that if $\cal A$ is an abelian category
satisfying certain mild conditions, then one can introduce the concept of a moduli space
of (semi)stable objects which has the structure of a projective algebraic variety.

In Section \ref{s:appl} we explore several examples of relevant
abelian categories from representation theory to which we can
apply Theorem \ref{mthm}. The examples include highest weight
categories with finite poset of simple objects (blocks of the BGG
category $\cal O$ for instance), the category of Harish-Chandra
bimodules, and the category of bimodules of a finite-dimensional
Jordan algebra. We also give an example of a limit construction
which enables us to define a (scheme theoretic) moduli space
structure on a highest weight category whose underlying poset is
infinite (Example \ref{e:limit}).  However, none of the examples above
correspond to hereditary algebras, but rather, they correspond to
quasi-hereditary or stratified algebras. This fact provides strong
motivation to the study of the geometric structure of the moduli
spaces of (semi)stable representations for such algebras. Some interesting examples of
highest weight categories with finite poset coming from
finite-dimensional representation theory of current algebras have
just appeared in \cite{CJ}. We also
mention   paper \cite{GS} where it was shown  that the category of
bounded modules over the symplectic Lie algebra $sp(2n)$ is
equivalent to the category of weight modules for the $n$-th Weyl
algebra, whose blocks are just module categories of some quivers
with relations \cite{BBF}.
%---------------------------------------------------------------

\section{Categorical Stability Theory}\label{s:catst}

Let $\cal A$ be an abelian category and denote by $K_0(\cal A)$ its Grothendieck group.
We will write $V\in\cal A$ meaning $V\in{\rm Ob}(\cal A)$. The isomorphism class of $V$, as well as its image in $K_0(\cal A)$,
will be denoted by $[V]$.
Given $\gamma\in K_0(\cal A)$, set $\cal A_\gamma = \{V\in\cal A:[V]=\gamma\}$,
and let $\widetilde{\cal A}_\gamma$ be the additive closure of the full subcategory of $\cal A$ whose objects consists of all of the
subquotients of elements in $\cal A_\gamma$.
Clearly, $\widetilde{\cal A}_\gamma$ is a full abelian subcategory of $\cal A$.
Observe that if $\cal B$ is another abelian category and $\cal F:\cal B\to \cal A$ is an exact functor,
then $\cal F$ induces an abelian group homomorphism $K_0(\cal B)\to K_0(\cal A)$ also denoted by $\cal F$.

%%%%%%%%%%%%%%%%%%%%%%%%%

\subsection{Definition and Basic Properties}

The following definition was first proposed by Rudakov (see \cite{Ru} and
the more recent paper \cite{GKR}), probably inspired by \cite[Definition 1.1]{K}.

\begin{definition}\label{st}
A {\em stability structure} on $\cal A$ consists of a preorder $\preceq$ on the set of
objects of $\cal A$ satisfying the following properties:
\begin{itemize}
\item[(i)] trichotomy: for any two non-zero objects $A$ and $B$, either
$A\prec B$, or $B\prec A$, or $A\asymp B$;
\item[(ii)] seesaw property: for each short exact sequence $0 \to A \to B \to C \to 0$
of non-zero objects we have that:
$$ {\rm either}~~~ A\prec B \Longleftrightarrow
B\prec C \Longleftrightarrow A\prec C ~~,$$
$$ {\rm or}~~~ A\succ B \Longleftrightarrow
B\succ C \Longleftrightarrow A\succ C ~~,$$
$$ {\rm or}~~~ A\asymp B \Longleftrightarrow
B\asymp C \Longleftrightarrow A\asymp C ~~.$$
\end{itemize}
where:
\begin{itemize}
\item $A\asymp B$ if $A\preceq B$ and $B\preceq A$;
\item $A\prec B$ if $A\preceq B$ but not $A\asymp B$;
\item $A\succ B$ if $B\prec A$.
\end{itemize} \end{definition}

Bridgeland has introduced the notion of stability on triangulated categories \cite{Br}.
Stability on abelian categories is also discussed by Joyce in \cite[Definition 4.1]{J}.

\begin{definition}
A nonzero object $B$ is said to be {\em stable} if every non-trivial sub-object
$A\subset B$ satisfies $A\prec B$. Equivalently, $B$ is stable if every non-trivial
quotient-object $B\to C\to 0$ satisfies  $B\prec C$.
A nonzero object $B$ is said to be {\em semistable} if every non-trivial sub-object
$A\subset B$ satisfies $A\preceq B$. Equivalently, $B$ is semistable
if every non-trivial quotient-object $B\to C\to 0$ satisfies $B\preceq C$.
\end{definition}

It is not difficult to see that for any stability structure every simple object is stable, and every stable object is indecomposable. The converse is not true in general.
Semistable objects may be either indecomposable or decomposable; a decomposable object
$A=\oplus_j A_j$ is semistable iff each $A_i$ is semistable and $A_i \asymp A_j$ for all
$i,j$.

We also have Schur's Lemma for stable objects: if $A$ is a stable object and
${\rm Hom}(A,A)$ is a finite dimensional vector space over an algebraically closed field
$\mathbb{F}$, then ${\rm Hom}(A,A)=\mathbb{F}$ \cite[Theorem 1]{Ru}. Other general properties
of stability structures and of (semi)stable objects can be found at Rudakov's original paper \cite{Ru}.

Every abelian category can be given a trivial stability structure, in which any two objects
satisfy $A\asymp B$. For such structure, an object is stable if and only if it is simple,
while every non-simple object is semistable. It is unclear however whether every abelian category can be provided with a nontrivial stability structure. Below, we show how nontrivial
stability structures can be constructed on a large class of abelian categories, see Example
\ref{e:exslope} below.

%%%%%%%%%%%%%%%%%%%%%%%%%%%%%%%%%%%%%%%%%%%%%%%%%

\subsection{Constructing Stability Structures}\label{s:slope}

Let $R$ be a totally ordered $\mathbb R$-vector space  satisfying
\begin{equation}\label{e:order1}
a\in \mathbb R, a>0, \quad \gb r\in R, \gb r>\gb 0 \qquad\Rightarrow\qquad a\gb r>\gb 0 \text{ and } -\gb r<\gb 0.
\end{equation}
Let also $d: K_0(\cal A) \to \mathbb R$ be an additive function satisfying $d([A])>0$ if $[A]\ne 0$ ($A\in {\rm Obj}(\cal A)$), and $c: K_0(\cal A) \to R$ be a linear map.  We call the ratio
$$ \sigma(A) = \frac{c([A])}{d([A])}$$
the {\em $(c:d)$-slope} of the object $A$.
Given any two non-zero objects $A$ and $B$, set
$$ A \preceq B ~~ \Longleftrightarrow ~~ \sigma(A) \leq \sigma(B)$$
It follows that $A\asymp B ~\Leftrightarrow~ \sigma(A)=\sigma(B)$, and
$A\prec B ~\Leftrightarrow~ \sigma(A)<\sigma(B)$. Note that:
$$\sigma(A) - \sigma(B) = \frac{1}{d([A])d([B])}\det \left( \begin{array}{cc} d([B]) & c([B]) \\ d([A]) & c([A]) \end{array} \right),$$
where
$$\det \left( \begin{array}{cc} d([B]) & c([B]) \\ d([A]) & c([A]) \end{array} \right):= d([B])c([A]) - d([A])c([B]).$$
Thus $A\prec B$ if and only if the determinant above is smaller than $\gb 0$, and $A\asymp B$
if and only if the determinant above is the zero vector.

We have (cf. \cite[Lemma 3.2]{Ru}):

\begin{proposition}\label{slope}
The pre--order $\preceq$ defined above gives rise to a stability structure on $\cal A$.
\end{proposition}
\begin{proof}
The first axiom of Definition \ref{st} is easy to verify. To check the
seesaw property, consider the exact sequence $0\to A\to B\to C\to 0$, so that
$[B]=[A]+[C]$. and,  hence, $c([B])=c([A])+c([C])$ and $d([B])=d([A])+d([C])$. It follows that:
\begin{align*}
\det &\left( \begin{array}{cc} d([B]) & c([B]) \\ d([A]) & c([A]) \end{array} \right)=
\det \left( \begin{array}{cc} d([A])+d([C]) & c([A])+c([C]) \\ d([A]) & c([A]) \end{array} \right)=\\
&=\det \left( \begin{array}{cc} d([C]) & c([C]) \\ d([A]) & c([A]) \end{array} \right)=
\det \left( \begin{array}{cc} d([C]) & c([C]) \\ d([A])+d([C]) & c([A])+c([C]) \end{array} \right)=\\
&=\det \left( \begin{array}{cc} d([C]) & c([C]) \\ d([B]) & c([B]) \end{array} \right).
\end{align*}
Therefore, $\sigma(A)=\sigma(B)\Leftrightarrow\sigma(A)=\sigma(C)\Leftrightarrow\sigma(B)=\sigma(C)$ and
$\sigma(A)<\sigma(B)\Leftrightarrow\sigma(A)<\sigma(C)\Leftrightarrow\sigma(B)<\sigma(C)$, which by definition imply the seesaw property.
\end{proof}

\begin{example}[Jordan-H\"older categories]\label{e:exslope}\rm
Let $\cal A$ be an abelian category all of whose objects are of finite length, $\Lambda$ be the set of isomorphism classes of simple objects, and $[A:\lambda]$ be the multiplicity of $\lambda\in\Lambda$ in the object $A$. Fix an order on $\Lambda$ for which there exists a minimal element and consider the $\mathbb R$--vector space $R$ with basis $\Lambda$ equipped with the lexicographic order, which obviously satisfy \eqref{e:order1}.
Given functions $g: \Lambda\to \mathbb R$ with $g(\lambda)>0$ for all $\lambda$, and $f:\Lambda\to \mathbb R$,  set
\begin{equation*}
c([A]) = \sum_{\lambda\in\Lambda} f(\lambda)[A:\lambda]\lambda \quad\text{and}\quad d([A]) =\sum_{\lambda\in\Lambda}g(\lambda)[A:\lambda].
\end{equation*}
Hence, every Jordan-H\"older category can be equipped with a stability structure.\hfill\qedsymbol
\end{example}

\begin{example}[Gieseker stability for torsion-free sheaves]\rm
Fix an $n$-dimensional projective variety $X$ over $\mathbb{C}$, and let ${\rm TF}(X)$
be the quasi-abelian category of torsion-free coherent sheaves on $X$. Given a coherent
sheaf $E$ on $X$, its {\em Hilbert polynomial} is defined as:
$$ p_E(t) = \chi(E(t)) = \sum_{p=1}^{n} (-1)^{p}\cdot \dim H^p(X,E(t)) ~~, $$
where as usual $E(t)=E\otimes\ox(t)$. This is a polynomial on $t$
with rational coefficients of degree at most $n$, so-called Hilbert
polynomial of $E$. It defines an additive function $p: K_0({\rm
TF}(X)) \to P_n(\mathbb{Q})$, where $P_n(\mathbb{Q})$ denotes the
$\mathbb{Q}$-vector space of polynomials of degree at most $n=\dim
X$. If $E$ is torsion-free, then $p_E(t)$ has degree exactly $n$
whenever $E$ is the nonzero sheaf; moreover, $p_E(t)=0$ if and only
if $E$ is the zero sheaf. We therefore can define an additive
function $r: K_0({\rm TF}(X)) \to \mathbb{Q}$, with $r(E)$ given the
leading coefficient of $p_E(t)$, which is always positive for every
nonzero sheaf. Providing $P_n(\mathbb{Q})$ with the lexicographic
order of the coefficients, we get that the {\em $(p:r)$-slope} slope
function yields a stability condition on ${\rm TF}(X)$, which is
known as the Gieseker stability. \hfill\qedsymbol
\end{example}

%%%%%%%%%%%%%%%%%%%%%%%%%%%%%%%%%%%%%%%%%

\subsection{Harder-Narasimhan Filtrations and Moduli Sets}\label{s:hn}

One of the main results for stability structures on abelian categories is the existence of a
{\em Harder-Narasimhan filtration} for any object of $\cal A$:

\begin{theorem} {\rm \cite[Theorems 2 and 3]{Ru}}\label{HN}
Let $\cal A$ is a noetherian abelian category provided with a stability condition $\prec$.
If $A$ is a semistable object, then there exists a unique filtration
$$ A = F^{n+1}A \hookleftarrow F^nA \hookleftarrow \cdots \hookleftarrow F^1A
\hookleftarrow F^0A = 0 ~~, $$
so-called Harder-Narasimhan filtration, such that:
\begin{itemize}
\item[(i)] the factors $Q_k=F^{k-1}A/F^kA$ are stable;
\item[(ii)] $Q_1\asymp Q_2\asymp \cdots \asymp Q_n$.
\end{itemize} 
\end{theorem}

\begin{definition}
Two semistable objects in ${\cal A}_\gamma$ are said to be S-equivalent if their
Harder-Narasimhan filtrations have the same composition factors.
\end{definition}

It is easy to see that S-equivalence is indeed an equivalence relation, and that two
stable objects are S-equivalent if and only if they are isomorphic.

\begin{definition}
Let  ${\cal A}_\gamma^s$  be the collection of all semistable objects within ${\cal A}_\gamma$.
The moduli set of semistable objects of $\cal A$ represented by the class $\gamma$
up to S-equivalence is given by ${\cal C}{\cal A}_\gamma^s={\cal A}_\gamma^s/\sim$,
where $\sim$ denotes S-equivalence.
\end{definition}

If the stability structure comes from a slope function $\sigma$ as above, we shall
write ${\cal A}_\gamma^\sigma$ and ${\cal C}{\cal A}_\gamma^\sigma$ instead of
${\cal A}_\gamma^s$ and ${\cal C}{\cal A}_\gamma^s$.

Under some circumstances, the moduli set ${\cal C}{\cal A}_\gamma^s$
has the structure of an algebraic variety. To see this, we will need
the concept of pull-back for a stability structure. Let $\cal A$ be
an abelian category provided with a stability structure $\preceq$.
Let $\cal B$ be an abelian category, and consider an exact functor
$\cal F:\cal B\to\cal A$. We can then pull-back the stability
structure in $\cal A$ to one in $\cal B$ in the obvious way: given
$U,W\in\cal B$ we declare $U\prec W$ if and only if ${\cal
F}(U)\prec {\cal F}(W)$. The following proposition is easily
verified.

\begin{proposition}\label{functor}
Suppose $\cal A, \cal B$, and $\cal F$ are as above, $\beta \in K_0(\cal B)$, and suppose
that $\cal F|_{\widetilde{\cal B}_\beta}: \widetilde{\cal B}_\beta\to
\widetilde{\cal A}_{\cal F(\beta)}$ is an equivalence of categories. Then $\cal F$ induces
a bijection $\overline{\cal F}:\cal C\cal B_\beta^s\to \cal C\cal A_{\cal F(\beta)}^s$.
\end{proposition}

%%%%%%%%%%%%%%%%%%%%%%%%%%%%%%%%%%%%%%%%%%%%%%%%%

\subsection{The Main Theorem}

We begin recalling the rephrasing of the notion of slope stability in terms of {\em character} stability.
Given a totally ordered $\mathbb R$-vector space $R$ satisfying \eqref{e:order1},
a $(c:d)$-slope $\sigma:K_0(\cal A)\to R$, and $\gamma\in K_0(\cal A)$, let
$\theta: K_0(\cal A)\to R$ be defined by
\begin{equation}\label{e:char}
\theta([V]) = -c([V]) +\sigma(\gamma)d([V]).
\end{equation}
Observe that if $V\in \cal A_{\gamma}$, then $\theta([V]) = 0$ and $V$ is stable
(resp. semistable) iff $\theta([U])>0$ (resp. $\theta([U])\ge 0$) for all
sub--object $U$ of $V$. Conversely, given abelian group homomorphisms
$\theta: K_0(\cal A)\to  R$ and $d: K_0(\cal A)\to  \mathbb R$, $\gb r\in R$,
and $\gamma\in K_0(\cal A)$ such that $\theta(\gamma)=0$ and $d(\gamma)\ne 0$, set
$$c([V]) = \theta([V]) + \frac{d[V]}{d(\gamma)}\ \gb r,$$
and let $\sigma$ be the corresponding slope. Then we see that, if $V\in\cal A_{\gamma}$, then
$V$ is stable (resp. semistable) iff $\theta([U])>0$ (resp. $\theta([U])\ge 0$) for all sub--object $U$ of $V$. Hence, if we restrict ourselves to $\cal A_{\gamma}$, for some $\gamma\in K_0(\cal A)$, slope stability can be defined by the choice of a character, i.e., an abelian group homomorphism $\theta: K_0(\cal A)\to  R$ satisfying $\theta(\gamma) = 0$.

From now on we will assume that $R$ is a totally ordered $\R$-vector space satisfying condition \eqref{e:order1}
and also that $R$ is equipped with a norm $\|\cdot\|$ such that
\begin{equation}\label{e:order2}
\gb r'<\gb r \Rightarrow\ \exists\ \varepsilon>0\quad\text{such that}\quad \gb r''<\gb r \quad\text{whenever}\quad |\gb r''-\gb r'\|\le \varepsilon
\end{equation}
We say that an order satisfying conditions (\ref{e:order1}) and (\ref{e:order2}) is {\em continuous} or that $R$ is a continuously ordered $\mathbb R$-vector space.

\begin{lemma}\label{l:vec<>scalar}
Let $R$ be a continuously ordered $\mathbb R$-vector space, $\cal A$ be an abelian category, and $\gamma\in K_0(\cal A)$ be such that $\widetilde{\cal A}_\gamma$ is a Jordan-H\"older category. Then for any $(c:d)$-slope function $\sigma:K_0(\cal A)\to R$,
there exists a character $\theta:K_0(\cal A)\to\R$ such that:
\begin{enumerate}
\item $\theta(\gamma)=0$ and $\theta([V])\in\Z$ for every simple object $V\in\widetilde{\cal A}_\gamma$.
\item An object $V\in\cal A_\gamma$ is $\theta$-semistable (resp. $\theta$--stable) iff it is $\sigma$-semistable (resp. $\sigma$--stable).
\end{enumerate}
\end{lemma}

\begin{proof}
We can assume, without loss of generality, that $\cal A=\widetilde{\cal A}_\gamma$. Let $\tilde\theta$ be the character defined by \eqref{e:char}.
Since $\tilde\theta(K_0(\cal A))$ is a finite--dimensional subspace of $R$, we can also suppose that $R = \mathbb R^k$ equipped with the usual inner product $\langle,\rangle$. It suffices to obtain $\theta$ such that $\theta(V)\in \mathbb Q$ for all simple $V\in\cal A$.

Let $I=\{1,\cdots,n\}$ be the index set of isomorphism classes of simple objects, $\gamma_i, i\in I$, be the corresponding images in $K_0(\cal A)$,  and define $\tilde\theta_i: K_0(\cal A)\to R$ by setting $\tilde\theta_i(\gamma_j) = \delta_{ij}\tilde\theta(\gamma_i)$, so that
$\tilde\theta =\sum_i \theta_i$. Choose a basis $\{\tilde\psi_1,\cdots,\tilde\psi_m\}\subseteq \{\tilde\theta_1, \cdots, \tilde\theta_n\}$ for the $\mathbb Q$-vector space generated by  $\{\tilde\theta_1, \cdots, \tilde\theta_n\}$. Thus, we have $\tilde\theta_i = \sum_j a_{ij}\tilde\psi_j$ for some $a_{ij}\in\mathbb Q$. For each choice of $\psi_1,\cdots,\psi_m\in \mathbb Q^k\subseteq R$, define
$\theta'(\gamma_i) = \sum_j a_{ij}\psi_j$ and extend by linearity. The linear independence of $\tilde\psi_j$ immediately implies $\theta'(\gamma)=0$. Since $\theta'$ depends continuously on the choice of $\psi_j$ and there are only finitely many $\beta\in K_0(\cal A)$ which can be the class of a sub-object of objects in $\cal A_\gamma$, it follows that we can choose $\tilde\psi_j$ so that an object $V\in\cal A_\gamma$ is $\theta'$--semistable (resp. $\theta'$--stable) iff it is $\tilde\theta$--semistable (resp. $\tilde\theta$--stable).
Finally, for each choice of $\bar\theta\in \mathbb Q^k$, set $\theta(\beta) = \langle\theta'(\beta),\bar\theta\rangle$. As before we can choose $\bar\theta$ so that we have {\it 2}.
\end{proof}

Now let $B$ be a finite dimensional associative algebra over an algebraically closed field
$\mathbb{F}$ of characteristic zero, and let ${\cal B}=$mod-$B$, the category of finite dimensional (left) $B$-modules. Given a class $\beta\in K_0(\cal B)$ and a group homomorphism $\theta:K_0(\cal B)\to\R$ with $\theta(\beta)=0$, it was shown by King in \cite{K} that the categorical moduli set ${\cal C}{\cal B}_\beta^\theta$ (i.e. the set of $\theta$-semistable objects up to S-equivalence) has the structure of a projective variety over $\mathbb{F}$. In light of this key example, we can conclude that:

\begin{theorem}\label{mthm}
Let $\cal A$ be an abelian category, $R$ be a continuously ordered $\mathbb R$-vector space,
$\cal B =$ {\rm mod}-$B$ for some finite-dimensional algebra $B$, and $\gamma\in K_0(\cal A)$.
Suppose $\gamma$ is such that there exists an exact functor ${\cal F}:\cal B \to{\cal A}$ such that $\gamma={\cal F}(\beta)$ for some $\beta\in K_0(\cal \beta)$ and
$\cal F|_{\widetilde{\cal B}_\beta}: {\widetilde{\cal B}_\beta}\to {\widetilde{\cal A}_\gamma}$ is an equivalence of categories.
Then, for any slope function $\sigma:K_0(\widetilde{\cal A}_\gamma)\to R$, the moduli set ${\cal C}{\cal A}_\gamma^\sigma$ can be given a structure of projective $\mathbb{F}$-variety.
\end{theorem}

\begin{proof}
By Lemma \ref{l:vec<>scalar} we can find a character
$\theta:K_0(\cal A)\to\mathbb R$ preserving (semi)stable objects in
$\cal A_\gamma$. Since $\cal F:K_0(\widetilde{\cal B}_\beta)\to
K_0(\widetilde{\cal A}_\gamma)$ is an isomorphism, we can regard
$\theta$ as a character on $\cal B$. From \cite[Proposition 4.3]{K}
we know that ${\cal C}{\cal B}_\beta^s$, where the stability
structure on $\cal B$ is the one pulled-back from $\cal A$, is a
projective variety. Hence, by Proposition \ref{functor}, this
structure can be transported to ${\cal C}{\cal A}_\gamma^\sigma$.
\end{proof}

For instance, if $\cal A$ is a noetherian abelian category with finitely many
nonisomorphic simple objects $V_1,\dots,V_n$ for which there are projective
covers $P_i\to V_i\to 0$, then $B={\rm End}(\oplus P_i)$ is a finite dimensional
associative algebra and $\cal A$ is equivalent to mod-$B$. In this context, Theorem
\ref{mthm} applies, and we conclude that ${\cal C}{\cal A}_\gamma^\sigma$ has the structure
of a projective variety.

%---------------------------------------------------------------

\section{Applications}\label{s:appl}

If $B$ in Theorem \ref{mthm} is a hereditary algebra, then more can be said about
${\cal C}{\cal A}_\gamma^\sigma$: it is irreducible, normal, and has dimension
equal to $\dim\left({\rm Ext}_\cal A^1(V,V)\right)$, for $V$ being a generic object
represented by $\gamma$; the subset of isomorphism classes of stable objects is open
(in the Zariski topology) and nonsingular \cite{K}. Furthermore, the birational type
of these varieties is studied in \cite{Sch}.

However, many natural categories that arise in
representation theory are equivalent to module categories of finite-dimensional
algebras which are not hereditary, but say quasi-hereditary or standardly stratified.
In this subsection we present some examples of such categories.
It would be interesting to work out more detailed algebraic geometric
properties of ${\cal C}{\cal A}_\gamma^\sigma$ for them.

%%%%%%%%%%%%%%%%%%%%%%%%%%%%%%%%%%%%%%%%%%%%%%%%%

\subsection{Highest Weight Categories and Quasi-Hereditary Algebras}

Let $\mathcal A$ be an $\mathbb F$-linear category. Following \cite{CPS} we
say that $\mathcal A$ is locally artinian if it is closed under direct limit (union) and  every  object is a union of objects of finite length.
We also assume that $\mathcal A$ has enough injective modules and that $V\cap \left(\cup_i U_i\right)=\cup_i (V\cap U_i)$
for any collection of objects $V, \{U_i\}$. An object $W$ is said to be a composition
factor of an object $V$ if it is a composition factor of a finite-length subobject of $V$. In this case we denote by
$[V:W]$ the supremum of the multiplicities of $W$ in all such subobjects.

A locally artinian category $\mathcal A$ is called a
\emph{highest weight category} \cite{CPS} if there exists an
interval-finite poset $\Lambda$ (i.e. the sets $[\lambda,
\mu]=\{z:\lambda\leq z\leq \mu\}$ are finite) such that:
\begin{enumerate}
\item Non isomorphic  simple objects $\{L(\lambda), \lambda\in
\Lambda\}$ in $\mathcal A$ are parameterized by $\Lambda$.
\item For each $\lambda\in \Lambda$ there exists $\Delta(\lambda)\in\cal A$ and a monomorphism $L(\lambda)\hookrightarrow \Delta(\lambda)$ such
that any composition factor $L(\mu)$ of $\Delta(\lambda)/L(\lambda)$ satisfies $\mu<\lambda$.
\item For every $\lambda,\mu\in\Lambda$,  $\dim \left({\rm Hom}(\Delta(\lambda), \Delta(\mu))\right)$ and $[\Delta(\lambda):L(\mu)]$ are finite.
\item The injective envelope $I(\lambda)$ of $L(\lambda)$ has a filtration $0=N_0\subset N_1\subset \ldots$ such that
$I(\lambda)=\cup_i N_i$, $N_1\simeq \Delta(\lambda)$,
$N_n/N_{n-1}\cong L(\mu)$ for some $\mu=\mu(n)>\lambda$ and given $\mu\in \Lambda$ there exist only finitely many $n$ for which
$\mu=\mu(n)$.
\end{enumerate}
The elements of $\Lambda$ are called weights and condition $2$ above explains the terminology ``highest weight category''. Note that $L(\lambda)$ is the socle of $\Delta(\lambda)$. Let
${\mathcal A}_f$ be the full (artinian) subcategory in $\mathcal
A$ consisting of objects of finite length.

\begin{theorem}(\cite{CPS}, Theorem 3.5)\label{cpshwfda}
If the poset $\Lambda$ is finite then the category ${\mathcal
A}_f$ is equivalent to a module category for some
finite-dimensional algebra.
\end{theorem}

In particular, since $\cal A_f$ is a Jordan-H\"older category, it follows from Theorem \ref{mthm} that for every $\gamma\in K_0(\cal A_f)$ and any slope function $\sigma$ with values on a continuously ordered $\mathbb R$-vector space, the moduli set $\cal C\cal A^\sigma_\gamma$ can be given a structure of projective variety.

Finite-dimensional algebras that correspond to highest weight
categories are called \emph{quasi-hereditary}. Such algebra $B$
can be characterized by the existence of a \emph{hereditary chain}
of two-sided ideals
$$0\subset J_1\subset J_2\subset \ldots \subset J_t\subset B,$$ where
$J_s/J_{s-1}=(B/J_{s-1})e_s(B/J_{s-1})$ for some idempotent $e_s$
of $B/J_{s-1}$ for all $1\leq s\leq t$ \cite{Dl1}.

If $\cal A$ is a highest weight category then there are two
canonical ways of constructing new highest weight categories. Let
$\Lambda$ be the indexing poset for $\cal A$, $\Upsilon\subset
\Lambda$ a proper nonempty subset, $\Omega=\Lambda\setminus
\Upsilon$. Denote by $\widetilde{\cal A}_{\Upsilon}$ the full subcategory of
$\cal A$ consisting of objects $M$ having composition factors
isomorphic to some $L(\mu)$, $\mu\in \Upsilon$. Since $\widetilde{\mathcal A}_{\Upsilon}$ is a Serre subcategory of $\cal A$ then the
quotient category $\cal A({\Omega})=\cal A /\widetilde{\mathcal A}_{\Upsilon}$ is defined. We say that $\Upsilon$ is an ideal of
$\Lambda$ if $\mu< \nu$ and $\nu\in \Upsilon$ implies $\mu\in
\Upsilon$. The set $\Upsilon$ is said to be finitely generated
provided $\Upsilon$ is the union of finitely many intervals $\{\mu:\mu< \nu\}$ for some $\nu\in \Upsilon$. We have the following
recollement property of highest weight categories.

\begin{theorem} (\cite{CPS}, Theorem 3.9)\label{t50}
If $\Upsilon$ is a finitely generated ideal of $\Lambda$ then
$\widetilde{\cal A}_{\Upsilon}$ is a highest weight category. Moreover, if
$\Omega$ is a finite coideal then $\cal A({\Omega})$ is a
highest weight category.
\end{theorem}

Let $\Lambda$ be a finite poset and $A$ be a quasi-hereditary
algebra corresponding to the category $\cal A_f$. Denote by
$e$  a complete sum of primitive idempotents representing
$\Upsilon$. Then the category $\widetilde{\cal A}_{\Upsilon, f}$ is equivalent
to the module category of $eAe$ if $\Upsilon$ is an ideal, hence
$eAe$ is quasi-hereditary. Moreover, $A/eAe$ is quasi-hereditary
if $\Omega$ is coideal of $\Lambda$.

\begin{example}[Poset Categories {\cite[Example 6.9]{PS}}]\rm
Let $(P, \leq)$ be a finite poset, $|P|$ the geometric realization
of the simplicial complex associated with $P$,
$P_{n+1}=P_n\setminus  \{$maximal elements of $P_n\}$, $P_0=P$.
Denote by $Sh(|P|)$ the category of sheaves of $\mathbb F$-vector spaces
on $|P|$ which are locally constant on the natural strata
$|P_n|\setminus |P_{n+1}|$. Then $Sh(|P|)$ is a highest weight
category which is isomorphic to the module category of the poset
$(P, \leq)$.\hfill\qedsymbol
\end{example}

The next example is the key example that originated the theory of
highest weight categories.

\begin{example}[Category $\cal O$]\label{e:cato}\rm

Let $\lie g$ be a complex  finite-dimensional simple Lie algebra, $\lie h$ be a fixed Cartan subalgebra, $\omega_i$ be the fundamental weighs,  $P,Q$ the weight and root lattices of $\lie g$, respectively. Let also $Q^+$ be the submonoid of $Q$ generated by the positive roots and recall that $\lie h^*$ can be equipped with the partial order $\mu\le\lambda$ iff $\lambda-\mu\in Q^+$. The BGG category $\cal O(\lie g)$ \cite{BGG} is the $\mathbb C$-linear category whose objects consists of $\lie g$-modules $V$ satisfying:
\begin{enumerate}
\item $V=\opl_{\mu\in\lie h^*}^{} V_\mu$, where $V_\mu= \{v\in V: hv=\mu(h)v \ \forall \ h\in\lie h\}$.
\item $\dim(V_\mu)$ is finite for all $\mu\in\lie h^*$.
\item There exist $\lambda_1,\cdots,\lambda_m\in\lie h^*$ (depending on $V$) such that $V_\mu\ne 0$ implies $\mu\le \lambda_j$ for some $j=1,\cdots,m$.
\end{enumerate}
The elements $\mu\in\lie h^*$ such that $V_\mu\ne 0$ are called the weights of $V$. It turns out that every block of $\cal O$ is equivalent to a block  $\cal A$ whose weights of its objects lie  in $P$. Moreover, in such a block there exists a unique $\lambda_0\in P$ such that the weight of the objects in $\cal A$ are bounded by $\lambda_0$ from above and there exists $V\in\cal A$ such that $V_{\lambda_0}\ne 0$. $\cal A$ is a highest weight category with poset $\Lambda$ given by $\Lambda = \{w\cdot \lambda_0\}$ where $w$ runs in the Weyl group $\cal W$ of $\lie g$ and the action of $\cal W$ in $\lie h^*$ is the so-called shifted action. In particular, $\Lambda$ is finite and, by Theorem \ref{cpshwfda}, $\cal A$ is equivalent to the module category of a (quasi hereditary) finite-dimensional algebra. If $\lambda_0$ is anti-dominant, then $\Lambda$ is a singleton and $\cal A$ is a semisimple category with a unique simple object, but, in general $\cal A$ is of wild representation type.  The objects $\Delta(\lambda),\lambda \in \Lambda$, are the so-called restricted duals of the Verma modules $M(\lambda)$.

Let us look more closely at some natural stability structure on $\cal A$. Let $R=\lie h^*_{\mathbb R}\subset \lie h^*$ be the $\mathbb R$-span of the fundamental weights and equip $R$ with the lexicographic ordering determined by a choice of ordering the nodes of the Dynkyn diagram of $\lie g$ and the usual order on $\mathbb R$.

Given an object $M$ in $\cal A$, fix $\gb x = (x_0,\cdots,x_{m-1})\in\mathbb Q^m$, where $m=|\Lambda|$, and define the slope
$$\sigma_{\gb x}([M]) = \frac{\sum_j x_j[M:L(\lambda_j)]\lambda_j}{\sum_j [M:L(\lambda_j)]}.$$
This defines a stability structure on $\cal A$ satisfying all the hypothesis of Theorem \ref{mthm}.

We now compute some examples of (semi) stable objects. We start with $\lie g =\lie{sl}_2$, even
though $\cal A$ is always of finite representation type in this case. In fact, beside the semisimple block, there is only one more class of
non equivalent blocks which is equivalent to mod-$B$, where $B$ is the 5-dimensional algebra whose underlying quiver is:

\begin{picture}(0.00,40.00)
\put(165.00,27.00){$a$}
\put(165.00,6.50){$b$}
\put(220.00,17.00){$ab=0.$}
\put(145,17.00){$_1\circ$}
\put(185.00,17.00){$\circ_2$}
\put(156.00,22.00){\vector(1,0){27.00}}
\put(183,18){\vector(-1,0){27.00}}
\end{picture}

\noindent The principal block (the one containing the trivial representation) is a representative of this class. We have $m=2,\lambda_0=0$, and  5 indecomposable objects: $L(0), L(-2)$ (the 2 irreducible objects), the Verma module $M(0)$, its
restricted dual $M^*(0)=\Delta(0)$, and the big projective
$P(-2)$. Here we have identified $P$ with $\mathbb Z$ by sending $\omega_1$ to 1. There are 4
non-split exact sequences:
\begin{gather*}
0\to L(-2)\to M(0)\to L(0)\to 0\\
0\to L(0)\to M^*(0)\to L(-2)\to 0\\
0\to L(-2)\to P(-2)\to M^*(0)\to 0\\
0\to M(0)\to P(-2)\to L(-2)\to 0
\end{gather*}

and the right regular representation of $A$ has the following structure

$$
\begin{array}{c}
2\\1\end{array}\ \oplus\  \begin{array}{c}
1\\2\\1\end{array}\ =\ A_A, 
$$
where $1$ stands for $L(-2)$ and $2$ stands for $L(0)$.

Now we compute the slopes:
\begin{gather}
\sigma(L(\lambda))= 0, \qquad\sigma(L(-2)) = -2x_2,\\
\sigma(M(0)) = \sigma(M^*(0)) =-x_2, \quad \sigma(P(-2)) = -\frac{4}{3}x_2
\end{gather}
If $x_2> 0$ then $M(0)$ is stable, while $M^*(0)$ is stable if and
only if $x_2<0$. On the other hand, $P(-2)$ is semistable if and
only if $x_2=0$, in which case the stability structure is trivial,
i.e. $V\asymp W$ for all $V,W\in\cal A$.

Now let us look at $\lie g=\lie{sl}_3$. This time the principal
block is wild so we do not attempt to characterize the semistable
objects completely. If $\cal A$ is the principal block then
$\lambda_0=0$ and
$$\Lambda = \{0,-2\omega_1+\omega_2,\omega_1-2\omega_2, -3\omega_2, -3\omega_1, -2\omega_1-2\omega_2\}=\{\lambda_0,\cdots,\lambda_5\}.$$
It is not difficult, but a bit tedious, to see that for all
$\lambda_k\in \Lambda$, there always exists a choice of $x_j,
j=1,\cdots, 5$, such that $M(\lambda_k)$ is stable.

We also remark that the blocks of the following subcategory of $\cal O$  are also highest weight categories.
Given $\lie p\subseteq \lie g$ a parabolic subalgebra, let $\mathcal O(\lie g, \lie p)$ be the
subcategory of $\mathcal O(\lie g)$ whose objects split into a sum of finite-dimensional modules of $\lie p$.
Such categories were studied in \cite{Ro}.\hfill\qedsymbol
\end{example}

The highest weight category of the next example is not equivalent to the module category of a finite-dimensional algebra. However it is the limit of some ``truncated'' highest weight subcategories to which Theorem \ref{cpshwfda} applies.

\begin{example}[A Limit Construction]\label{e:limit}\rm
Let $\lie g$ be an affine Kac-Moody algebra with a Cartan
subalgebra $\lie h$, $Q$ the root lattice and $Q^+$ the sub-monoid of $Q$
generated by the  positive roots, and $\pi=\{\alpha_1, \ldots, \alpha_n\}$
be a set of simple roots. Define the height of
$\eta = \sum_{i=1}^n k_i\alpha_i\in Q^+$ by $|\eta|=\sum_i k_i$.

Categories of truncated $\lie g$-modules were studied in \cite{RW}.
For $k\in \mathbb Z_+$, set $Q^+(k)=\{\eta\in Q^+:|\eta|>k\}$.
If $\eta=\sum_{i=1}^n k_i\alpha_i\in Q$, set $\eta^+=\sum_{j:k_j\in\mathbb Z_+} k_j\alpha_j$ and, given  $\lambda\in \lie h^*$,
$k\in \mathbb Z_+$, denote by $\Pi=\Pi(\lambda, k)$ the set of all $\mu\in \lie h^*$ such that
$(\mu - \lambda)^+\in Q^+\setminus Q^+(k)$.
Let also $\cal A(\lambda, k)$ be the full category of the category of all $\lie g$-modules consisting of those modules $V$ such that:
\begin{enumerate}
\item $V=\opl_{\mu\in\lie h^*}^{} V_\mu$, where $V_\mu= \{v\in V: hv=\mu(h)v \ \forall \ h\in\lie h\}$.
\item $\dim(V_\mu)$ is finite for all $\mu\in\lie h^*$.
\item $V_{\mu}=0$ for all $\mu\in \lie h^*\setminus \Pi$.
\end{enumerate}

Clearly, any simple object of the category $\cal A(\lambda,k)$ is
a quotient of the corresponding Verma module. We will denote an
irreducible module with highest weight $\mu$ by $L(\mu)$. It
follows from the results of \cite{RW} that $\cal A(\lambda,k)$ is
a highest weight category as well as its full subcategory $\cal
A_f(\lambda, k)$ consisting of finite-length objects. The later
category has infinitely many simple objects and hence is not a
module category for a finite dimensional algebra. Nevertheless,
this category can be ``approximated'' by certain
finite-dimensional algebras.  Indeed, consider the following
stratification of the category $\cal A_f(\lambda, k)$. Given a
positive integer $m$ denote by $\cal A_f(\lambda, k)^m$ the full
subcategory in $\mathcal A_f(\lambda, k)$ consisting of those
objects whose composition factors are isomorphic to $L(\mu)$ for
some $\mu\le \lambda$ with $|\lambda-\mu|\le m$. Let $\Omega$ be
the set of all $\mu\in \Pi$ with $|\lambda-\mu|\leq m$ (we set
$|\lambda - \mu|=0$ if $\lambda<\mu$). Then clearly $\Omega$ is a
coideal in $\Lambda$ and hence, $\cal A_f(\lambda, k)^m$ is a
highest weight category by Theorem \ref{t50}. Moreover, since
$\Omega$ is finite then $\cal A_f(\lambda, k)^m$ is equivalent to
the module category for some quasi-hereditary algebra.

We have an embedding of full subcategories
$$\cal A_f(\lambda, k)^1\subseteq \ldots \subseteq \cal A_f(\lambda, k)^m\subseteq \ldots$$
and $$\cal A_f(\lambda, k)=\lim_{\longrightarrow}\mathcal A_f(\lambda, k)^m.$$

To shorten the notation, set $\cal A=\cal A_f(\lambda,k)$ and $\cal A(m)=\cal A_f(\lambda, k)^m$.
Fix a slope function $\sigma:K_0(\cal A)\to R$, where $R$ is an $\R$-vector space with a continuous total order and let $\gamma\in K_0(\cal A)$. Then $\gamma = \sum_\mu k_\mu [L(\mu)]$ for some nonnegative integers $k_\mu$, all but finitely many equal to zero. Let 
$$m_0 = \max \{|\lambda -\mu| : k_\mu\ne 0\}.$$ It follows that $\cal A_\gamma^\sigma = \cal A(m_0)_\gamma^\sigma$ has a structure of quasi-projective variety. 
\end{example}

The results of this section are valid for a larger class of
algebras introduced in \cite{RW}, which includes symmetrizable
Kac-Moody algebras and the Witt algebra.

%%%%%%%%%%%%%%%%%%%%%%%%%%%%%%%%%%%%%%%%%%%%%%%%%

\subsection{Stratified Categories and Algebras}

In this section we consider examples of categories equivalent to
module categories of finite-dimensional algebras which are a
certain generalization of quasi-hereditary algebras.

Let $A$ be a finite-dimensional algebra and $\Lambda$ be a poset
parameterizing the isomorphism classes of simple  $A$-modules.
Then $A$ is called \emph{$\Delta$-filtered} (\cite{ADL},
\cite{Dl1}, \cite{Dl2}) with respect to the right (respectively
left) module structure if the following conditions are satisfied:

\begin{itemize} 
\item[(i)] There exists a collection of right (resp. left) $A$-modules
$\Delta(\lambda), \lambda\in \Lambda$, such that $\Delta(\lambda)$
has a unique irreducible quotient $L(\lambda)$ corresponding to
$\lambda$ and all other composition factors $L(\mu)$ satisfy
$\mu\leq \lambda$; 
\item[(ii)] The right (resp. left) regular
representation of $A$ is filtered by $\Delta(\lambda)$'s, i.e. for
each $\lambda\in \Lambda$, $\Delta(\lambda)$ is a homomorphic
image of the projective cover $P(\lambda)$ of $L(\lambda)$ and the
kernel has a finite filtration with factors $\Delta(\mu)$ with
$\lambda<\mu$.
\end{itemize}

If $A$ has either  right or left $\Delta$-filtration then it is
also called \emph{standardly stratified} (\cite{FKM1}). Note that
strandardly stratified algebras are special cases of stratified
algebras introduced in \cite{CPS1}. If $A$ has both left and right
filtrations then it is called \emph{properly stratified}. Finally,
a properly stratified algebra is quasi-hereditary if and only if
the choice of $\Delta(\lambda)$'s is the same for right and left
regular representations.

One way to obtain new standardly stratified algebras out of a
given projectively stratified (or even quasi-hereditary) algebra
$A$ is by taking $eAe$ for some idempotent $e$. Notice that even
if $A$ is quasi-hereditary, $eAe$ may not be so. Here we just list
some examples, for details see \cite{Ma}.

\begin{example}[Harish-Chandra Bimodules]\rm
Let $\lie g$ be a complex finite-dimensional simple Lie algebra and $\theta$
a central character. The category $\mathcal H$ of Harish-Chandra bimodules
consists of finitely generated algebraic $U(\lie g)$-bimodules.
Denote by $\mathcal H(\theta)$ the full subcategory of
modules with central character $\theta$ (with respect to the right action of the center).
Then $\mathcal H(\theta)$ is equivalent to a certain subcategory of $\cal O$ \cite{BG}
which in turn is equivalent to the category mod-$eAe$ for some quasi-hereditary algebra
$A$ and some idempotent $e$.\hfill\qedsymbol
\end{example}

\begin{example}[Parabolic Category $\cal O$]\rm
A parabolic generalization of the category $\cal O$ which contains
non-highest weight irreducible modules was defined in \cite{CF}
and studied extensively in \cite{FKM1}, \cite{FKM2}, \cite{FKM3}.
This category corresponds to a fixed parabolic subalgebra of a
simple finite-dimensional Lie algebra. The role of Verma modules
is played by the \emph{generalized Verma modules}. In \cite{FKM4}
this theory was extended to the case of infinite-dimensional Lie
algebras with triangular decomposition.\hfill\qedsymbol
\end{example}

Note that suitable blocks of the categories mentioned in these two
examples are equivalent. Some other equivalences of these
categories were exploited in \cite{KM}.

%%%%%%%%%%%%%%%%%%%%%%%%%%%%%%%%%%%%%%%%%%%%%%%%%

\subsection{Jordan Algebras}

Now we present an example of a category  equivalent to the module category of a finite-dimensional algebra which is neither hereditary, nor quasi-hereditary, nor
stratified.

Recall, that a Jordan $\mathbb F$-algebra is an $\mathbb F$-vector space $J$ with
a binary operation ``$\cdot$'' compatible with the vector space structure such that
any $\,a,b\in J\,$
\begin{equation} \begin{array}{c}
a\cdot b =b \cdot a \\ ((a\cdot a)\cdot b)\cdot a= (a\cdot a)\cdot (b\cdot a).
\end{array} \end{equation}

It is known that the category of finite-dimensional $J$-bimodules
is equivalent to the category of left finite-dimensional modules
over some associative algebra $U(J)$ which is called the universal
multiplicative envelope of $J$. If $J$ is finite-dimensional then
$U(J)$ is finite-dimensional as well and, hence, the whole
stability machinery can be applied to the category of
$J$-bimodules. If $J$ is a semisimple Jordan algebra then $U(J)$
is semisimple and the  variety of semistable objects is
trivial. On the other hand, if $J$ is not semisimple but
$Rad^2(J)=0$, then the category of bimodules might even be of wild
representation type and, thus, have very nontrivial moduli sets.

We point out that the algebra $U(J)$ decomposes into a product of
subalgebras $U(J)=U_0\oplus U_{\frac{1}{2}}\oplus U_1$, where
$U_0=\mathbb F$ and $U_0\oplus U_{\frac{1}{2}}$ is  the special
universal envelope of $J$. Then the module category $U(J)$-mod is
equivalent to a direct sum
$$U_0\text{-mod}\oplus U_{\frac{1}{2}}\text{-mod}\oplus U_1\text{-mod}$$
and the problem is reduced to the study of these module categories. For
detailed study of the category $U_0\oplus U_{\frac{1}{2}}-mod$ see \cite{KOS}.

%%%%%%%%%%%%%%%%%%%%%%%%%%%%%%%%%%%%%%%%%%%%%%%%%

\subsection{Finite-Dimensional Representations of Loop Algebras}

We end the section with an example of an abelian category which is
a very active research topic and is not equivalent to a module
category of a finite-dimensional algebra. Since it is a
Jordan-H\"older category, we can construct the moduli sets.
However, nothing can be said about geometric properties of these
moduli sets at the moment. This is the category of
finite-dimensional representations of the loop algebras $\tlie
g=\lie g\otimes \mathbb C[t,t^{-1}]$, where $\lie g$ is a complex
finite-dimensional simple Lie algebra. We remark that, although
this category is not a highest weight category, it is a very close
cousin of the highest weight category studied in \cite{CJ}.

Since every finite-dimensional $\tlie g$-module is in category $\cal
O(\lie g)$ when regarded as a $\lie g$-module, we can define slope
functions in exactly the same manner as we did in Example
\ref{e:cato}. However, here we have other natural choices for the
function $d$. Instead of using $\lie g$-length as in Example
\ref{e:cato}, we can use $\tlie g$-length or even dimension.
However, we do not expect such choice to significantly change the
structure of the ``generic'' moduli sets.


\begin{thebibliography}{99}
\bibitem{ADL} I.Agoston, V.Dlab, E.Lukacs, Stratified algebras,
C.R.Math.Rep. Acad.Sci. Canada 20 (1998), 20-25.

\bibitem{BBF}  V.Bekkert, G.Benkart, V.Futorny, Weight modules for
Weyl algebras, in ``Kac-Moody Lie algebras and related topics'',
Contemp. Math 343 Amer. Math. Soc., Providence, RI, 2004, 17-42.

\bibitem{BG}
I. Bernstein, I. Gelfand,
Tensor product of finite and infinite-dimensional representations of semisimple Lie algebras,
Compositio Math. 41 (1980), no.2, 245-285.

\bibitem{BGG}
I.N.~Bernstein, I.M.~Gel'fand, S.I.~Gel'fand,
Category of $\mathfrak{g}$ modules,
Funct. Anal. Appl. {\bf 10} (1976), 87--92.

\bibitem{Br}
T. Bridgeland,
Stability conditions on triangulated categories.
To appear in Annals of Math.
Preprint math.AG/0212237.

\bibitem{CJ}
V. Chari, J. Greenstein,
Current algebras, highest weight categories and quivers, Preprint math.RT/0612206.

\bibitem{CPS}
E. Cline, B. Parshall, L. Scott,
Finite dimensional algebras and highest weight categories,
J.Reine Agne. Math. 391 (1988), 85-99.

\bibitem{CPS1}
E. Cline, B. Parshall, L. Scott,
Stratifying endomorphism algebras,
Mem. Amer. Math. Soc. 124 (1996), n.591.

\bibitem{CF} A.Coleman, V.Futorny, Stratified $L$-modules,
J.Algebra 163 (1994), no.1, 219-234.

\bibitem{Dl1}
V. Dlab, Quasi-hereditary algebras revisited, An. St. Univ.
Ovidius Constantza 4 (1996), 43-54.

\bibitem{Dl2}
V. Dlab,
Properly stratified algebras,
C. R. Acad. Sci. Paris Ser. I Math 331 (2000) 191-196.

\bibitem{FKM1}
V. Futorny, S. Koenig, V. Mazorchuk,
Catgeories of induced modules and standardly strtified algebras,
Algebr. Represent. Theory 5 (2002), no.3, 259-276.

\bibitem{FKM2}
V. Futorny, S. Koenig, V. Mazorchuk,
A combinatorial description of blocks in $\cal O(P,\Lambda)$ associated with
$sl(2)$-induction,
J. Algebra 231 (2000), no.1, 86-103.

\bibitem{FKM3}
V. Futorny, S. Koenig, V. Mazorchuk,
S-subcategories in $\cal O$,
Manuscripta Math. 102 (2000), no.4, 487-503.

\bibitem{FKM4}
V. Futorny, S. Koenig, V. Mazorchuk,
Categories of induced modules for Lie algebras with triangular decomposition,
Forum Math. 13 (2001), 641-661.


\bibitem{GS} D.Grantcharov, V.Serganova, Category of
$sp(2n)$-modules with bounded weight multiplicities, Moscow Math.
J. to appear.


\bibitem{GKR}
A. Gorodentscev, S. Kuleshov, A. Rudakov,
Stability data and t-structures on a triangulated category.
Izv. Math. {\bf 68}, 749-781 (2004).
%Preprint math.AG/0312442.

\bibitem{J}
D. Joyce,
Configurations in abelian categories, III. Stability conditions and invariants.
Preprint math.AG/0410267.

\bibitem{K}
A. King,
Moduli of representations of finite dimensional algebras.
Quart. J. Math. {\bf 45}, 515-530 (1994).

\bibitem{KM}
S. Koenig, V. Mazorchuk,
Enright's completions and injectively copresented modules.
Preprint.

\bibitem{KOS}
I. Kashuba, S. Ovsienko, I. Shestakov,
Jordan algebras versus associative.
Preprint RT 2005-10, University S\~ao Paulo.

\bibitem{Ma}
V. Mazorchuk,
Stratified algebras arising in Lie theory.
U.U.D.M. Report 2002:37.

\bibitem{PS}
B. Parshall, L. Scott,
Derived categories, quasi-hereditary algebras, and algebraic groups,
Math. Lecture Series, n.3, Carleton Univ. (1988), 1-105.

\bibitem{Ro}
A. Rocha-Caridi,
Splitting criteria for $\lie g$-modules induced from a parabolic and the
Bernstein-Gelfand-Gelfand resolution of a finite dimensional
irreducible $\lie g$-module,
Trans. Amer. Math. Soc. 262 (1980), 335-366.

\bibitem{RW}
A. Rocha-Caridi, N. Wallach,
Projective modules over graded Lie algebras. I,
Math. Z. 180 (1982), 151-177.

\bibitem{Ru}
A. Rudakov,
Stability on abelian categories.
J. Alg. {\bf 197}, 231-245 (1997).

\bibitem{Sch}
A. Schofield,
Birational classification of moduli spaces of representations of
quivers.
Indag. Mathem. {\bf 12}, 407-432 (2001).

\end{thebibliography}
\end{document}